\documentclass[10pt]{amsart}
\usepackage{amsmath,amssymb,amsthm, enumerate, cite}
\usepackage{theoremref}

\title[Conjugacy class conditions in l.c.s.c groups]{Conjugacy class conditions in locally compact second countable groups}
\author{Phillip Wesolek} 

\address{Universit\'{e} catholique de Louvain,
Institut de Recherche en Math\'{e}matiques et Physique (IRMP),
Chemin du Cyclotron 2, box L7.01.02,
1348 Louvain-la-Neuve,
Belgique}
\email{phillip.wesolek@uclouvain.be}

\def\cprime{$'$} \def\cprime{$'$}
\providecommand{\bysame}{\leavevmode\hbox to3em{\hrulefill}\thinspace}
\providecommand{\MR}{\relax\ifhmode\unskip\space\fi MR }

\providecommand{\href}[2]{#2}

\newcommand{\N}{\mathbb{N}}
\newcommand{\ol}[1]{\overline{#1}}

\newcommand{\acts}{\curvearrowright}
\newcommand{\sleq}{\leqslant}
\newcommand{\sgeq}{\geqslant}

\newtheorem{thm}{Theorem} [section]
\newtheorem{lem}[thm]{Lemma}
\newtheorem{cor}[thm]{Corollary}

\newtheorem{fact}[thm]{Fact}

\theoremstyle{definition}

\newtheorem{rmk}[thm]{Remark}
\newtheorem*{claim}{Claim}

\newtheorem*{ack}{Acknowledgments}
\begin{document}

\begin{abstract}
Many non-locally compact second countable groups admit a comeagre conjugacy class. For example, this is the case for $S_{\infty}$, $Aut(\mathbb{Q},<)$, and, less trivially, $Aut(\mathcal{R})$ for $\mathcal{R}$ the random graph [Truss]. A. Kechris and C. Rosendal ask if a non-trivial locally compact second countable group can admit a comeagre conjugacy class. We answer the question in the negative via an analysis of locally compact second countable groups with topological conditions on a conjugacy class. 
\end{abstract}

\maketitle
\let\thefootnote\relax

\section{Introduction}
Our goal is to answer a question of Kechris and Rosendal: Can a non-trivial locally compact second countable group admit a comeagre conjugacy class? By an unpublished argument due to K.H. Hofmann, such a group cannot be connected. This, of course, leaves the problematic totally disconnected locally compact case. Here the theory is not nearly as well developed. We therefore begin with an analysis of \emph{totally disconnected locally compact second countable} (t.d.l.c.s.c.) groups.
 
\begin{thm}\thlabel{non-meagre} 
Suppose $U$ is a second countable profinite group and $g\in U$. The following are equivalent:
\begin{enumerate}
\item $g^U:=\{ugu^{-1}\;|\;u\in U\}$ is non-meagre.
\item $g^U$ is open.
\item $\mu(g^U)>0$ where $\mu$ is the normalized Haar measure on $U$.
\item There is an integer $M>0$ such that $|C_{U/V}(gV)|\sleq M$ for every open normal subgroup $V\trianglelefteq U$.
\end{enumerate}
\end{thm}
\noindent Recall $C_{U/V}(gV)$ denotes the centralizer of $gV$ in $U/V$. We extend \thref{non-meagre} to the non-compact case.

\begin{cor}Suppose $G$ is a t.d.l.c.s.c. group and $g\in G$ is such that $cl(\langle g\rangle)$ is compact. The following are equivalent:
\begin{enumerate}
\item $g^G$ is open.
\item $g^G$ is non-meagre.
\item $g^G$ is non-null.
\end{enumerate}
\end{cor}

Applying the above results and two deep theorems in profinite group theory, we eliminate t.d.l.c.s.c. groups as possible candidates for a positive answer to the motivating question.

\begin{thm}\thlabel{dense}
If $G$ is a non-trivial t.d.l.c.s.c. group and $g^G$ is dense, then $g^G$ is meagre and Haar null.
\end{thm}

We conclude by presenting Hofmann's proof in the connected case to give a complete answer to the question of Kechris and Rosendal.

\begin{thm}[Hofmann]
A non-trivial connected locally compact group cannot have a dense conjugacy class.
\end{thm}

\begin{thm} 
A non-trivial locally compact second countable group does not admit a comeagre conjugacy class. 
\end{thm}

Along the way, we present an example due to Rosendal showing the existence of a infinite profinite group with a non-meagre conjugacy class, and we sketch the example built by E. Akin, E. Glasner, and B. Weiss of a non-trivial t.d.l.c.s.c. group with a dense conjugacy class.

\begin{ack} 
The results herein form part of the author's thesis work; the author thanks his thesis adviser Christian Rosendal. The author also thanks the Hausdorff Research Institute for Mathematics in Bonn for its hospitality; this paper was brought to a polished form during a stay at the Institute. The author finally thanks the anonymous referee for his or her detailed remarks.
\end{ack}

\section{Preliminaries} 
We fix a few conventions and notations for this paper. All topological groups are assumed to be Hausdorff, and all subgroups are taken to be closed unless otherwise noted. To indicate $O$ is an open subgroup of a topological group $G$, we write $O\sleq_oG$. For $g\in G$ and $A\subseteq G$, $g^A:=\{aga^{-1}\;|\;a\in A\}$. We use the abbreviations l.c., t.d., and s.c., for ``locally compact", ``totally disconnected", and ``second countable", respectively.

\subsection{L.c.s.c. groups}
L.c.s.c. groups are $K_{\sigma}$ and metrizable by classical results \cite{K95}. These groups have a canonical left invariant Borel measure which is unique up to constant multiples called the \emph{Haar measure}. When a subset of a group is said to be \emph{non-null}, we mean with respect to the Haar measure. A familiarity with the Haar measure and basic properties thereof is assumed; \cite{DE09} contains a nice, brief introduction.\par

\indent L.c.s.c. groups are connected-by-totally disconnected. Indeed, for $G$ a l.c.s.c. group let $G^{\circ}$ denote the connected component of the identity. It is easy to see $G^{\circ}$ is a closed normal subgroup and there is a short exact sequence of topological groups
\[
1\rightarrow G^{\circ}\rightarrow G\rightarrow G/G^{\circ}\rightarrow 1
\]
where $G^{\circ}$ is connected and $G/G^{\circ}$ is totally disconnected. The study of l.c.s.c. groups thus reduces to the study of connected l.c.s.c. groups and t.d.l.c.s.c. groups. \par

\indent The study of connected locally compact groups reduces to the study of inverse limits of Lie groups by the celebrated solution to Hilbert's fifth problem.

\begin{thm}[Gleason, Montgomery, Yamabe, Zippin, see \cite{MZ55}]\thlabel{hilberts5}
A connected locally compact group is pro-Lie. 
\end{thm}

A group $G$ is a \emph{Lie group} if $G$ has an analytic $\mathbb{K}$-manifold structure such that the map $(g,h)\mapsto gh^{-1}$ is analytic where $\mathbb{K}$ is either $\mathbb{R},\mathbb{C}$ or some non-discrete ultrametric field. If a Lie group $G$ is connected and locally compact, then $\mathbb{K}$ is either $\mathbb{R}$ or $\mathbb{C}$ and $G$ is finite dimensional. Associated to a Lie group $G$ is the \emph{Lie algebra}, $\mathfrak{g}$, where $\mathfrak{g}$ is the tangent space at the identity along with a bracket operation. The self action of $G$ by conjugation induces an action of $G$ on $\mathfrak{g}$ by vector space isomorphisms. When $G$ is connected and locally compact, this action is given by the map $Ad:G\rightarrow GL(\mathfrak{g})$ where $GL(\mathfrak{g})=GL_n(\mathbb{K})$ for some finite $n$ and $\mathbb{K}$ equal to $\mathbb{R}$ or $\mathbb{C}$. 

\begin{fact}[{\cite[III.6.4 Corollary 4]{Bo98}}]\thlabel{adjoint}
If $G$ is a connected locally compact Lie group, then $Ad:G/Z(G)\rightarrow im(Ad)$ is an isomorphism of Lie groups.
\end{fact}
\noindent Proofs of the aforementioned properties of Lie groups may also be found in \cite{Bo98}.\par

\indent For t.d.l.c. groups, a central theorem is an old result of D. van Dantzig.

\begin{thm}[van Dantzig {\cite[II.7.7]{HR79}}]\thlabel{vanDantzig}
A t.d.l.c. group admits a basis at the identity of compact open subgroups.
\end{thm}

Elements of a t.d.l.c. group that lie in a compact open subgroup are called \emph{periodic}. Note that periodic elements do not necessarily have finite order; we call an element \emph{torsion} if it has finite order.  For a t.d.l.c. group $G$, we denote the collection of periodic elements of $G$ by $P_1(G)$. It is easy to see
\[
P_1(G)=\{g\in G\;|\; cl(\langle g \rangle)\text{ is compact }\}.
\]

\subsection{Profinite groups} The compact open subgroups of a t.d.l.c. group given by van Dantzig's theorem are \emph{profinite}. I.e. they are inverse limits of finite groups. An introduction to the theory of profinite groups may be found in the texts \cite{RZ00} and \cite{Wil98}. We assume a familiarity with profinite groups and merely recall a few relevant definitions and theorems.\par

\indent Profinite groups have a basis at $1$ consisting of open normal subgroups. For a second countable profinite group $U$, we say $(N_i)_{i\in \mathbb{N}}$ is a \emph{normal basis at $1$} for $U$ if each $N_{i}$ is open and normal in $U$, $\bigcap_{i \in \mathbb{N}}N_{i}=\{1\}$, and $(N_{i})_{i\in \N}$ is $\subseteq$-decreasing.\par

\indent The \emph{Frattini subgroup} of a profinite group $U$, denoted $\Phi(U)$, is the intersection of all maximal proper open subgroups. The Frattini subgroup is the collection of \emph{non-generators} of $U$: $x\in U$ is a non-generator if whenever $U=\ol{\left<X\cup \{x\}\right>}$ for $X\subseteq U$, then $U=\ol{\left<X\right>}$. This implies a useful observation: If $H\sleq U$ is closed and $H\Phi(U)=U$, then $H=U$. We note an easy and illuminating proof: Suppose $H\sleq U$ is closed and $H\Phi(U)=U$. If $H\neq U$, then there is some maximal proper open $W\lneq U$ such that $H\sleq W$. However, then $H\Phi(U)\sleq W$ contradicting that $H\Phi(U)=U$. We thus conclude $H=U$.\par

\indent  When $U$ is pro-$p$, an inverse limit of $p$-groups, $\Phi(U)$ has a well understood structure.

\begin{fact}[{\cite[Lemma 2.8.7]{RZ00}} ]\thlabel{frat}
If $U$ is pro-$p$, then $\Phi(U)=U^p[U,U]$ where $[U,U]$ is the closure of the commutator subgroup and $U^p$ is the closed subgroup generated by all $p$ powers.
\end{fact}

We also note two deep results in profinite group theory.

\begin{thm}[Zel'manov \cite{Z92}] \thlabel{zel}
Every torsion pro-$p$ group is locally finite.
\end{thm}

\begin{thm}[Wilson \cite{Wil83}]\thlabel{wil} 
Let $U$ be a compact Hausdorff torsion group. Then $U$ has a finite series 
\[
 \{1\}=U_0\sleq U_1\sleq\dots\sleq U_n=U
 \]
of closed characteristic subgroups in which each factor $U_i/U_{i-1}$ is either \textbf{(1)} pro-$p$ for some prime $p$ or \textbf{(2)} isomorphic to a Cartesian product of isomorphic finite simple groups.
\end{thm}

\section{Non-meagre conjugacy classes}
We first consider t.d.l.c.s.c. groups with a non-meagre conjugacy class. The key step is to initially consider profinite groups with a \emph{non-null} conjugacy class.

\begin{lem} \thlabel{keylem}
Let $U$ be a profinite group with normalized Haar measure $\mu$. If $h\in U$ is such that $\mu(h^U)>0$, then 
\begin{enumerate}
\item For all $N\trianglelefteq_oU$, $|C_{U/N}(hN)|\sleq \frac{1}{\mu(h^U)}$.
\item $|C_{U}(h)|\sleq \frac{1}{\mu(h^U)}$, and in particular, $h$ is torsion.
\end{enumerate} 
\end{lem}
\begin{proof}
Let $h\in U$ have a non-null conjugacy class and take $N\trianglelefteq_o U$. Take a minimal set of coset representatives $h,k_1hk_1^{-1},\dots,k_nhk_n^{-1}$ for $(hN)^{U/N}$ in $U/N$. Certainly, 
\[
h^U\subseteq hN\cup\dots\cup k_nhk_n^{-1}N,
\]
so 
\[
0<\mu(h^U)\sleq |(hN)^{U/N}|\mu(N).
\]
Since $1=\mu(U)=|U/N|\mu(N)$, we have
\[
\begin{array}{ccl}
 |(hN)^{U/N}|\mu(N) & = & |U/N:C_{U/N}(hN)|\mu(N)\\
 		&  = & \frac{|U/N|\mu(N)}{|C_{U/N}(hN)|} \\
 		& = & \frac{1}{|C_{U/N}(hN)|}.
\end{array}
\] 
Hence, $|C_{U/N}(hN)|\sleq\frac{1}{\mu(h^U)}$, and it follows $|C_U(h)|\sleq \frac{1}{\mu(h^U)}$.
\end{proof}

Via \thref{keylem}, category and measure theoretic notions of size for conjugacy classes in second countable profinite groups agree.

\begin{thm}\thlabel{open conj} Suppose $U$ is a second countable profinite group and $g\in U$. The following are equivalent:
\begin{enumerate}
\item $g^U$ is non-meagre.
\item $g^U$ is open.
\item $\mu(g^U)>0$ where $\mu$ is the normalized Haar measure on $U$.
\item There is $M>0$ such that $|C_{U/N}(gN)|\sleq M$ for all $N\trianglelefteq_oU$.
\end{enumerate}
\end{thm}
\begin{proof}
$(1)\Rightarrow (2)$ Since $g^U$ is non-meagre, it is somewhere dense. However, $g^U$ is closed since the continuous image of a compact set, and therefore, $g^U$ contains an non-empty open set. Say $O\subseteq g^U$. So
\[
g^U=\bigcup_{u\in U}uOu^{-1},
\]
and $g^G$ is open.\par

\indent $(2)\Rightarrow (3)$ is immediate from properties of the Haar measure. \par

\indent For $(3)\Rightarrow (4)$, take $N\trianglelefteq_o U$. By \thref{keylem},
\[
|C_{U/N}(gN)|\sleq \frac{1}{\mu(g^U)}.
\] 
Fixing $M\sgeq \frac{1}{\mu(g^U)}$, we have $(4)$. \par

\indent $(4)\Rightarrow (1)$ Fix $(N_i)_{i\in \N}$ a normal basis at $1$ for $U$ and take $M>0$ which witnesses $(4)$ for $g$. For each $i$, let $A_i\subseteq g^U$ be a minimal set of coset representatives for $(gN_i)^{U/N_i}$ in $U/N_i$.\par

\indent For each $i\in \N$,
\[
g^U\subseteq A_iN_i:=\bigcup_{a\in A_i}aN_i,
\]
and it is easy to check
\[
\bigcap_{i\in \N}A_iN_i=g^U.
\]
Since $(A_iN_i)_{i\in \N}$ is an $\subseteq$-decreasing sequence,
\[
\mu(A_iN_i)\rightarrow \mu(g^U)
\]
by continuity from above for $\mu$. \par

\indent On the other hand, $|C_{U/N_i}(gN_i)|\sleq M$ for any $i\in \N$, and therefore, $|C_{U/N_i}(gN_i)|$ only takes on finitely many values as $i$ varies. By passing to a subsequence, we may assume $|C_{U/N_i}(gN_i)|=k\sleq M$ for all $i$. So
\[
\mu(A_iN_i)=|(gN_i)^{U/N}|\mu(N_i)=\frac{1}{|C_{U/N_i}(gN_i)|}=\frac{1}{k}
\]
for each $i$, and it follows $\mu(A_iN_i)= \mu(g^U)$ for all $i$.\par

\indent We now fix an $i$ and consider
\[
E:=(A_iN_i)\setminus g^U.
\] 
Certainly $E$ is open and $\mu(E)=0$. As the only Haar null open set is $\emptyset$, $E=\emptyset$, and $A_iN_i=g^U$. We conclude $g^U$ is open and, a fortiori, non-meagre.
\end{proof}
\noindent We remark that $(4)$ of the above theorem is an algebraic characterization of a profinite group having an open conjugacy class.\par

\indent We now apply our results for second countable profinite groups to obtain a result for all t.d.l.c.s.c. groups.

\begin{lem}\thlabel{non-null} Let $G$ be a t.d.l.c.s.c. group. If $g^G$ is non-null, then $g^U$ is non-null for any compact open subgroup $U$ of $G$. If $g$ is also periodic, then $g$ is torsion and $g^U$ is open for any compact open subgroup $U$ containing $g$.
\end{lem}
\begin{proof} Fix $U\sleq_oG$ compact. Since $G$ is second countable, there is a countable set $(h_i)_{i\in \N}$ such that $G=\bigcup_{i\in\N}h_iU$. So $g^G=\bigcup_{i \in \N} g^{h_iU}$, and for some $i\in \N$, $\mu(g^{h_iU})>0$. Fix such an $i$; now
\[
0<\mu(g^{h_iU})=\mu(h_ig^Uh_i^{-1})=\mu(g^U)\Delta(h^{-1}_i)
\]
where $\Delta$ is the modular function. Since the modular function is strictly positive, $\mu(g^U)>0$ as desired. \par

\indent If $g$ is also periodic, take $W$ a compact open subgroup containing $g$. By the uniqueness of the Haar measure, $g^W$ is non-null in $W$ with respect to the normalized Haar measure on $W$. \thref{keylem} then implies $g$ is torsion, and \thref{open conj} implies $g^W$ is open. 
\end{proof}

\begin{cor}\thlabel{mcat}Suppose $G$ is t.d.l.c.s.c. group and $g$ is periodic. The following are equivalent:
\begin{enumerate}
\item $g^G$ is open.
\item $g^G$ is non-meagre.
\item $g^G$ is non-null.
\end{enumerate}
\end{cor}
\begin{proof}
$(1)\Rightarrow (2)$ This follows by the Baire category theorem. \par

$(2)\Rightarrow (3)$ Suppose $g^G$ is non-meagre. Since $G$ is $K_{\sigma}$, we may write $G=\bigcup_{i\in \N}K_i$ with the $K_i$ compact, so $g^G=\bigcup_{i\in \N}g^{K_i}$. The Baire category theorem implies $g^{K_i}$ is non-meagre for some $i$; fix such an $i$. We have that $g^{K_i}$ is also closed and, thereby, has non-empty interior. It follows $\mu(g^G)>0$ proving $(3)$. \par

$(3)\Rightarrow (1)$ This is immediate from \thref{non-null}.
\end{proof}

It is not clear there are non-discrete examples of the groups discussed in this section. For completeness, we present an example due to C. Rosendal of a non-discrete second countable profinite group with a non-null conjugacy class. The author wishes to express his thanks to Rosendal for allowing this example to be included in the present work.

\subsection{Rosendal's example} Let $D_{6}$ be the dihedral group of the triangle. Recall
\[
D_6=\langle s,r\;|\; r^3=1,\hspace{2pt} s^2=1,\text{ and }sr=r^{-1}s\rangle.
\]
Every element of $D_6$ is of the form $r^i$ or $sr^i$ for $i=0,1,2$. Form $D_6^{\mathbb{N}}$ and define $G\sleq D_6^{\mathbb{N}}$ to be the collection of $\alpha$ such that all coordinates of $\alpha$ are of the form $sr^i$ or all coordinates are of the form $r^i$. Certainly, $G$ is a closed subset of $D_6^{\mathbb{N}}$ and is closed under inverses. It is easy to check $G$ is also closed under multiplication. \par

\indent Let $H\trianglelefteq G$ be the collection of $\alpha\in G$ for which all coordinates are of the form $r^i$. So $H$ is closed and index two in $G$, whereby $H$ is also open. Now consider the element $\beta\in G$ which is constantly equal to $s$. 

\begin{claim} 
Every $\gamma \in G\setminus H$ is conjugate to $\beta$ by an element of $H$. 
\end{claim}
\begin{proof} Let $\gamma \in G\setminus H$ and say $\gamma (i)=sr^{j(i)}$ where $j(i)\in \{0,1,2\}$. Define $\eta \in H$ as follows:
$$
\eta(i) =
\begin{cases}
1 & \text{if }j(i)=0 \\
r^2 & \text{if }j(i)=1\\
r & \text{if }j(i)=2
\end{cases}.
$$
One checks $\eta \gamma \eta^{-1}=\beta$. 
\end{proof}
\noindent By the claim, $G\setminus H$ is the conjugacy class of $\beta$, so $\beta^G$ is open and non-null.

\begin{rmk} We conclude this section with two remarks.
\begin{enumerate}
\item \thref{open conj} gives a measure and category equivalence which may be independently useful. For example, \thref{open conj} seems potentially useful to answer the following open question asked by L. L\'{e}vai and L. Pyber in \cite{LP00}: Let $U$ be profinite and put $T_n:=\{x\in U\;|\; x^n=1\}$ for $n\sgeq 2$. If $\mu(T_n)>0$, then is $T_n$ non-meagre?  For $n=2$, the question is known to have a positive answer \cite{LP00}. 

\item Rosendal's example is solvable. It is unknown if all such examples must be virtually solvable. The following partial results are known to the author: $(i)$ Second countable pronilpotent groups with an open conjugacy class are solvable. $(ii)$ Second countable profinite groups $U$ such that $g\in U$ has an open conjugacy class and $|g|$ is prime are virtually solvable. 
\end{enumerate}
\end{rmk}

\section{Dense conjugacy classes}
In this section, we consider a much different topological condition on a conjugacy class: density.

\begin{lem}\thlabel{pro-p}
A torsion pro-$p$ group with an open conjugacy class is finite 
\end{lem}
\begin{proof} Suppose $U$ is a torsion pro-$p$ group with an open conjugacy class $h^U$. Since $h^{-1}h^U\subseteq [U,U]$, we have that $[U,U]$ is open, and $\Phi(U)$ is open by \thref{frat}. Let $u_1,\dots,u_k$ be coset representatives for $\Phi(U)$ in $U$. Plainly,
\[
U=cl(\langle u_1,\dots,u_n\rangle)\Phi(U),
\]
so $U=cl(\langle u_1,\dots,u_n\rangle)$ since $\Phi(U)$ is the collection of non-generators. Zel'manov's theorem, \thref{zel}, now implies $U$ is finite.
\end{proof}

\begin{thm}\thlabel{dense conj null}
If $G$ is a non-trivial t.d.l.c.s.c. group and $g^G$ is dense, then $g^G$ is meagre and null. 
\end{thm}
\begin{proof} 
Suppose toward a contradiction $G$ is a non-trivial t.d.l.c.s.c. group and $g^G$ is dense and either non-meagre or non-null. Since $g^G$ is dense, $g\in P_1(G)$ by \thref{vanDantzig}. \thref{mcat} now implies $g^G$ is open, and by \thref{non-null}, $g$ is torsion. Say $|g|=n$ and observe $G$ must have exponent $n$. \par

\indent Consider $U$ a compact open subgroup of $G$. Since $U$ is torsion, there is a series of closed characteristic subgroups
\[
\{1\}=U_0\sleq U_1\sleq\dots\sleq U_n=U
\] 
given by Wilson's theorem, \thref{wil}. Let $k<n$ be greatest such that $U_k$ is not open in $U$. Since $U_{k+1}$ is open, $g^G$ meets $U_{k+1}$; without loss of generality, $g\in U_{k+1}$. So \thref{non-null} implies $g^{U_{k+1}}$ is open, and therefore, $U_{k+1}/U_k$ also has an open conjugacy class.\par

\indent By \thref{pro-p}, $U_{k+1}/U_k$ cannot be pro-$p$ since else $U_{k}$ is finite index and, therefore, open. We conclude $U_{k+1}/U_{k}$ must be isomorphic to a Cartesian product of isomorphic finite simple groups. Say $U_{k+1}/U_k\simeq \prod_{i\in I}S_i=:S$ and let $(s_i)_{i\in I}=s\in S$ have an open conjugacy class. \thref{keylem} implies $s$ has a finite centralizer. However,
\[
C_{S}(s)=\prod_{i\in I} C_{S_i}(s_i),
\]
and each $C_{S_i}(s_i)$ contains at least two elements. So $S$ must be a finite product, and $U_k$ is again open. We have thus contradicted the choice of $k$.
\end{proof}

\begin{cor}\thlabel{comeagre t.d.}
If $G$ is a non-trivial, t.d.l.c.s.c. group, then $G$ does not admit a comeagre  or co-null conjugacy class.
\end{cor}

\begin{rmk} \thref{comeagre t.d.} shows that a non-discrete t.d.l.c.s.c. analogue of an infinite discrete group with two conjugacy classes, e.g. \cite{HNN49}, is impossible. The main result of this paper shows indeed there is no non-trivial such l.c.s.c. group
\end{rmk}

We note the hypotheses of \thref{dense conj null} are not vacuous. Indeed, Akin, Glasner, and Weiss have built an example of a non-trivial t.d.l.c.s.c. group with a dense conjugacy class in \cite{AGW08}. We include a sketch of their example for completeness.

\subsection{The example of Akin, Glasner, and Weiss}
Let $\mathcal{J}=\{J_i\;|\;i\in \N\}$ be a sequence of non-empty finite subsets of $\mathbb{N}$ which partition $\N$ and have strictly increasing cardinality. Let $J^k:=\bigcup_{i=0}^kJ_i$ and $K_n$ be the collection of permutations in $Sym(\N)$ which setwise stabilize each of $J^n,J_{n+1},J_{n+2}\dots$; $K_n$ is the collection of permutations of $\N$ which preserve the partition beyond the $n$-th part. It is easy to see $K_n$ is compact as a subset of $Sym(\N)$ and $K_n\sleq_o K_{n+1}$. Put $G:=\bigcup_{i\in \N}K_i$ and give $G$ the inductive topology: $A\subseteq G$ is open if and only if $A\cap K_i$ is open in $K_i$ for all $i$. Akin, Glasner, and Weiss show $G$ is a t.d.l.c.s.c. group. Further, the sets
\[
G(\pi):=\{g\in K_n\;|\;g\upharpoonright_{J^n}=\pi\}
\]
where $n\in \N$ and $\pi\in Sym(J^n)$ vary form a basis for the topology on $G$.

\begin{claim} 
$G$ has a dense conjugacy class. 
\end{claim}
\begin{proof} It is enough to show $G$ is topologically transitive. I.e. for all basic open sets $G(\pi),G(\xi)\subseteq G$ there is $k\in G$ such that $kG(\pi)k^{-1}\cap G(\xi)\neq \emptyset$. Without loss of generality, we may assume $\pi, \xi \in Sym(J^n)$ for some $n$. Choose $k>n$ such that $|J_k|>|J^n|$; this is possible since the $J_i$ are strictly increasing in cardinality. Fix an injective map $\beta:J^n\rightarrow J_k$ and extend $\beta$ to an element $b$ of $G$ by 
$$
b(i) =
\begin{cases}
\beta(i) & i\in J^n \\
\beta^{-1}(i), & i\in \beta(J^n)\\
id, &\text{ else }
\end{cases}
$$
It is easy to check that there is $a\in G(\pi)$ such that $ab\upharpoonright_{J^n}=b\xi\upharpoonright_{J^n}$. We thus have $b^{-1}ab\in b^{-1}G(\pi)b\cap G(\xi)$ as desired.
\end{proof}

\begin{rmk}
It is worth nothing the above example has \emph{ample dense elements}: the diagonal action by conjugation of $G$ on the $n$-th Cartesian power of $G$ has a dense orbit for every $n\sgeq 1$.
\end{rmk}

\section{The non-existence of a comeagre conjugacy class}
By \thref{comeagre t.d.}, a non-trivial t.d.l.c.s.c. group does not admit a comeagre conjugacy class.  We now eliminate connected groups as candidates for admitting a comeagre class. The connected case is a previously unpublished result of Professor Hofmann. The author wishes to express his thanks to Professor Hofmann for permitting this result to be included in the present work. \par

\indent To present Hofmann's result, an old theorem due to W. Burnside is required.

\begin{fact}[Burnside {\cite[XVII.3 Corollary 3.3]{L02}}]\thlabel{burnside}
Let $E$ be a finite dimensional vector space over an algebraically closed field $k$ and $R$ a subalgebra of $End_k(E)$. If $E$ has no non-trivial proper $R$-invariant subspaces, then $R=End_k(E)$. We say $E$ is $R$-simple in such a case.
\end{fact}   

\begin{lem} \thlabel{no alg dense}
If $G\sleq GL_n(\mathbb{C})$ is a closed subgroup with a dense conjugacy class, then $G=\{Id\}$.
\end{lem}
\begin{proof}
We consider general linear groups to be written as matrix groups in the standard basis. Suppose $G\sleq GL_n(\mathbb{C})$ and say $A\in G$ has a dense conjugacy class. \par

\indent Since $A^G$ is dense, we may find $B_iAB_i^{-1}\rightarrow Id$ with $B_i\in G$. The determinant $\det:M_n(\mathbb{C})\rightarrow \mathbb{C}$ is continuous, and therefore,
\[
\det(tId-B_iAB_i)\rightarrow \det(tId-Id)=(t-1)^n.
\]
Since $\det$ is invariant under conjugation by elements of $GL_n(\mathbb{C})$, $\det(tId-B_iAB_i)=\det(tId-A)$ for all $i$, so $\det(tId-A)=(t-1)^n$. It follows every $B\in G$ has characteristic polynomial $(t-1)^n$. By the Cayley$-$Hamilton theorem, $(B-Id)^n=0$ for all $B\in G$. \par

\indent Consider $B,C\in G$ and let $Tr$ be the usual trace function. Then,
\begin{equation}
Tr(C(B-Id))=Tr(CB-Id)-Tr(C-Id)=0
\end{equation}
since the trace is linear and $Tr$ vanishes on elements for which some power is zero. Note $(1)$ holds for any linear group over $\mathbb{C}$ with a dense conjugacy class.\par

\indent We now consider $E$ a least dimension non-trivial $G$-invariant subspace of $\mathbb{C}^n$. Let $r:G\rightarrow GL(E)$ be the induced map. Certainly, $\tilde{G}:=cl(r(G))$ again has a dense conjugacy class, and $E$ is a $\tilde{G}$-simple vector space. Set $R\subseteq End_{\mathbb{C}}(E)$ to be the algebra generated by $\tilde{G}$. Since $R$ contains $\tilde{G}$, we see that $E$ is $R$-simple and $R=End_{\mathbb{C}}(E)$ by Burnside's theorem, \thref{burnside}. \par

\indent Fix $B\in \tilde{G}$ and take $M\in End_{\mathbb{C}}(E)$. Since $M=\sum_{i=1}^m\alpha_iA_i$ where $A_i\in \tilde{G}$ and $\alpha_i\in \mathbb{C}$, we have that $Tr(M(B-Id))=0$ by $(1)$. Consider $M_{ij}$ the matrix which is $1$ on the $(i,j)$-th entry and $0$ else. So $ M_{ij}(B-Id)$ is the matrix with diagonal consisting of zeros and the $(j,i)$-th entry of $(B-Id)$. In view of $(1)$, the $(j,i)$-th entry must be zero. We conclude $\tilde{G}$ is trivial, every element of $G$ fixes $E$, and $E$ is one dimensional. Since $E=:E_1$ is fixed by $G$, the action of $G$ on $\mathbb{C}^n$ gives an action on $\mathbb{C}^n/E_1$. We may thus repeat the argument above to obtain $E_1<E_2\dots <E_n=\mathbb{C}^n$ a strictly increasing sequence of $G$-invariant vector subspaces such that $E_{i+1}/E_i$ is one dimensional and $G$ acts trivially on $E_{i+1}/E_{i}$.\par

\indent Form $K_i:=\ker(G\acts E_i)$ for each $1\sleq i\sleq n$. Certainly, $K_1=G$. For $K_2$, $E_2$ has two dimensions; say $E_2=M\oplus E_1$. Fix $m+e\in E_2$, take $A,B\in G$, and consider their action on $m+e$. Since $G$ acts trivially on $E_1$ and $E_2/E_1$, $A(m+e)=m+\alpha e +e$ and $B(m+e)=m+\beta e +e$ for some $\alpha, \beta\in \mathbb{C}$. Additionally, $A^{-1}(m+e)=m-\alpha e +e$ and $B^{-1}(m+e)=m-\beta e +e$. These observations yield
\[
\begin{array}{ccl}
A^{-1}B^{-1}AB(m+e) & = & A^{-1}B^{-1}A(m+\beta e +e)\\
					& = & A^{-1}B^{-1}(m+\alpha e+\beta e +e)\\
					& = & A^{-1}(m-\beta e+\alpha e+\beta e +e)\\
					& = & (m-\alpha e - \beta e+\alpha e+\beta e +e)\\
					& =& m+e.
\end{array}
\]
Thus, $A^{-1}B^{-1}AB\in K_2$, $G/K_2$ is abelian, and $K_2=G$ since $G/K_2$ has a dense conjugacy class. Continuing in this fashion, $G=K_n$, and $G$ acts trivially on $\mathbb{C}^n$. We have thus demonstrated $G=\{1\}$.
\end{proof}

\begin{thm}[Hofmann]\thlabel{comeagre}
A non-trivial connected locally compact group cannot have a dense conjugacy class.
\end{thm}
\begin{proof}
For contradiction, suppose $G$ is a non-trivial connected locally compact group with a dense conjugacy class. By \thref{hilberts5}, $G$ is pro-Lie. We may thereby find a proper closed subgroup $N\triangleleft G$ such that $G/N$ is Lie. Since $N$ is proper, $G/N$ has a dense conjugacy class. \par

\indent Let $\tilde{G}:=Ad(G/N)\sleq GL_n(\mathbb{K})$ where $\mathbb{K}$ is either the real or complex field; under the natural inclusion $GL_n(\mathbb{R})\sleq GL_n(\mathbb{C})$, we may assume $\mathbb{K}=\mathbb{C}$. \thref{adjoint} implies $\tilde{G}$ has a dense conjugacy class, and therefore, $\tilde{G}=\{1\}$ by \thref{no alg dense}. So $G/N$ is trivial since abelian with a dense conjugacy class contradicting the choice of $N$. 
\end{proof}

Combining \thref{comeagre t.d.} and \thref{comeagre}, we answer Kechris' and Rosendal's question.

\begin{thm} 
A non-trivial l.c.s.c. group does not admit a comeagre conjugacy class. 
\end{thm}
\begin{proof} 
Suppose for contradiction $G$ is a l.c.s.c. group and $g^G$ is comeagre. Since $G$ is $K_{\sigma}$, we may write $G=\bigcup_{i\in \N}K_i$ with the $K_i$ compact sets. So $g^G=\bigcup_{i\in \N}g^{K_i}$, and there is $i\in \N$ such that $g^{K_i}$ is non-meagre by the Baire category theorem. For such an $i$, we see that $g^{K_i}$ is closed and, therefore, must have non-empty interior. It follows $g^G$ is open as well as dense.\par

\indent If $G$ is connected, we contradict \thref{comeagre}, so $\tilde{G}:=G/G^{\circ}$ must be non-trivial where $G^{\circ}$ is the connected component of the identity. However, the image of $g^G$ in $G/G^{\circ}$ is now an open dense conjugacy class in the non-trivial t.d.l.c.s.c. group $G/G^{\circ}$. This contradicts \thref{comeagre t.d.}, and we conclude the theorem.
\end{proof}


\def\cprime{$'$} \def\cprime{$'$}
\providecommand{\bysame}{\leavevmode\hbox to3em{\hrulefill}\thinspace}
\providecommand{\MR}{\relax\ifhmode\unskip\space\fi MR }
\providecommand{\MRhref}[2]{%
  \href{http://www.ams.org/mathscinet-getitem?mr=#1}{#2}
}
\providecommand{\href}[2]{#2}
\begin{thebibliography}{10}

\bibitem{AGW08}
E.~Akin, E.~Glasner, and B.~Weiss, \emph{Generically there is but one self
  homeomorphism of the {C}antor set}, Trans. Amer. Math. Soc. \textbf{360}
  (2008), no.~7, 3613--3630. \MR{2386239 (2008m:22009)}

\bibitem{Bo98}
N.~Bourbaki, \emph{Lie groups and {L}ie algebras. {C}hapters 1--3}, Elements of
  Mathematics (Berlin), Springer-Verlag, Berlin, 1998, Translated from the
  French, Reprint of the 1989 English translation. \MR{1728312 (2001g:17006)}

\bibitem{DE09}
A.~Deitmar and S.~Echterhoff, \emph{Principles of harmonic analysis},
  Universitext, Springer, New York, 2009. \MR{2457798 (2010g:43001)}

\bibitem{HR79}
E.~Hewitt and K.~Ross, \emph{Abstract harmonic analysis. {V}ol. {I}}, second
  ed., Grundlehren der Mathematischen Wissenschaften [Fundamental Principles of
  Mathematical Sciences], vol. 115, Springer-Verlag, Berlin, 1979, Structure of
  topological groups, integration theory, group representations. \MR{551496
  (81k:43001)}

\bibitem{HNN49}
G.~Higman, B.~H. Neumann, and H.~Neumann, \emph{Embedding theorems for groups},
  J. London Math. Soc. \textbf{24} (1949), 247--254. \MR{0032641 (11,322d)}

\bibitem{K95}
A.~Kechris, \emph{Classical descriptive set theory}, Graduate Texts in
  Mathematics, vol. 156, Springer-Verlag, New York, 1995. \MR{1321597
  (96e:03057)}

\bibitem{L02}
S.~Lang, \emph{Algebra}, third ed., Graduate Texts in Mathematics, vol. 211,
  Springer-Verlag, New York, 2002. \MR{1878556 (2003e:00003)}

\bibitem{LP00}
L.~L{\'e}vai and L.~Pyber, \emph{Profinite groups with many commuting pairs or
  involutions}, Arch. Math. (Basel) \textbf{75} (2000), no.~1, 1--7.
  \MR{1764885 (2001i:20059)}

\bibitem{MZ55}
D.~Montgomery and L.~Zippin, \emph{Topological transformation groups},
  Interscience Publishers, New York-London, 1955. \MR{0073104 (17,383b)}

\bibitem{RZ00}
L.~Ribes and P.~Zalesskii, \emph{Profinite groups}, second ed., Ergebnisse der
  Mathematik und ihrer Grenzgebiete. 3. Folge. A Series of Modern Surveys in
  Mathematics [Results in Mathematics and Related Areas. 3rd Series. A Series
  of Modern Surveys in Mathematics], vol.~40, Springer-Verlag, Berlin, 2010.
  \MR{2599132 (2011a:20058)}

\bibitem{Wil83}
J.~Wilson, \emph{On the structure of compact torsion groups}, Monatsh. Math.
  \textbf{96} (1983), no.~1, 57--66. \MR{721596 (85a:22007)}

\bibitem{Wil98}
\bysame, \emph{Profinite groups}, London Mathematical Society Monographs. New
  Series, vol.~19, The Clarendon Press Oxford University Press, New York, 1998.
  \MR{1691054 (2000j:20048)}

\bibitem{Z92}
E.~I. Zel{\cprime}manov, \emph{On periodic compact groups}, Israel J. Math.
  \textbf{77} (1992), no.~1-2, 83--95. \MR{1194787 (94e:20055)}

\end{thebibliography}
\end{document}